\documentclass[oneside,english]{amsart}

\usepackage{amssymb}
\usepackage{amscd}

\numberwithin{equation}{section} %% Comment out for sequentially-numbered
\numberwithin{figure}{section} %% Comment out for sequentially-numbered
\theoremstyle{plain}
\newtheorem*{thm*}{Theorem}
\theoremstyle{plain}
\newtheorem{thm}{Theorem}[section]
\theoremstyle{definition}

\theoremstyle{plain}
\newtheorem{lem}[thm]{Lemma}
\theoremstyle{plain}

\theoremstyle{plain}
\newtheorem{cor}[thm]{Corollary}
\theoremstyle{remark}

\theoremstyle{remark}
\newtheorem*{acknowledgement*}{Acknowledgement}

\begin{document}

\title[First integrals for Finsler metrics with vanishing $\chi$-curvature]{First integrals for Finsler metrics with vanishing $\chi$-curvature}

\author[Bucataru]{Ioan Bucataru}
\address{Faculty of  Mathematics \\ Alexandru Ioan Cuza University \\ Ia\c si, 
  Romania}
\email{bucataru@uaic.ro}
\urladdr{http://orcid.org/0000-0002-8506-7567}

\author[Constantinescu]{Oana Constantinescu}
\address{Faculty of  Mathematics \\ Alexandru Ioan Cuza University \\ Ia\c si, 
  Romania}
\email{oanacon@uaic.ro}
\urladdr{http://orcid.org/0000-0003-2687-2029}

\author[Cre\c tu]{Georgeta Cre\c tu}
\address{Department of  Mathematics \\ Gheorghe Asachi Technical University \\ Ia\c si, 
  Romania}
\email{cretuggeorgeta@gmail.com}
\urladdr{http://orcid.org/0000-0003-4197-0268}

\date{\today}

\begin{abstract}
We prove that in a Finsler manifold with vanishing $\chi$-curvature (in particular with constant flag curvature) some non-Riemannian geometric structures are geodesically invariant and hence they induce a set of non-Riemannian first integrals. Two alternative expressions of these first integrals can be obtained either in terms of the mean Berwald curvature, or as functions of the mean Cartan torsion and the mean Landsberg curvature. 
\end{abstract}



\maketitle

\section{Introduction.}

An important geometric structure that encodes informations about the non-Riemannian, Finslerian setting is given by the (mean) Cartan torsion. This structure has been used intensively to prove rigidity results in Finsler geometry, first by Akbar-Zadeh in \cite{AZ88}, and very recently by \'Alvarez-Paiva in \cite{AP21}. 

Another important non-Riemannian quantity in Finsler geometry is the so-called $\chi$-curvature, which has been introduced by Shen in \cite{Shen13}. Sprays with vanishing $\chi$-curvature were studied in a recent paper \cite{Shen21} by Shen. In Lemmas \ref{lemcfi} and \ref{lemchi} we show that Finsler metrics with vanishing $\chi$-curvature have some geodesically invariant structures: the mean Berwald curvature and a $1$-form $\alpha$ that depends on the mean Cartan torsion and the mean Landsberg curvature. We use these geodesically invariant structures to construct geodesically invariant functions and hence first integrals. One of these first integrals has been obtained previously in \cite[(1.1)]{BCC21} and it reduces to the first integral obtained by Foulon and Ruggiero for $k$-basic Finsler surfaces in \cite{FR16}. The fact that the $\chi$-curvature can be expressed in terms of the mean Cartan torsion is a key aspect for obtaining rigidity results in Finsler geometry. 
 
In this work we provide two alternative expressions for a set of $(n-1)$ non-Riemannian first integrals for Finsler metrics of vanishing $\chi$-curvature, in particular for Finsler metrics of  constant flag curvature. 

We fix our framework by considering a smooth $n$-dimensional manifold $M$, with $n\geq 2$, and $TM$ its
tangent bundle. A \emph{Finsler structure} is given by a function $F:TM\to [0, +\infty)$ that is continuous, positively $1$-homogeneous in the fiber coordinates, smooth on $T_0M=TM\setminus \{0\}$, and whose metric tensor: 
\begin{eqnarray}
g_{ij}(x,y)=\frac{1}{2}\frac{\partial^2F^2}{\partial y^i \partial y^j}(x,y) \label{gij}
\end{eqnarray}
is non-degenerate on $T_0M$. Here $(x,y)\in TM$, with $x\in M$ and
$y\in T_xM$.

On a Finsler manifold $(M,F)$, we consider $B^k_{ijl}$ the Berwald curvature, its mean Berwald curvature 
$ E_{ij}=\frac{1}{2}B^k_{ijk}$, given by \eqref{be}, and the $0$-homogeneous, $(1,1)$-type tensor: 
\begin{eqnarray}
\mathcal{E}^i_j=2Fg^{ik}E_{kj}.  \label{eij}
\end{eqnarray}
On a Finsler manifold $(M,F)$, the Finsler structure $F$ and the contravariant metric tensor $g^{ik}$ are covariant constant along the geodesic flow. Therefore, the tensor \eqref{eij} and all its geometric invariants will be constant along the geodesic flow if the mean Berwald curvature will also have this property. 

For the $(1,1)$-type tensor \eqref{eij}, we consider its characteristic polynomial
\begin{eqnarray}
P(\Lambda)=\det\left(\mathcal{E}^i_j+\Lambda \delta^i_j\right) = \sum_{a=0}^{n} g_a\Lambda^{n-a}=\Lambda^n+ g_1\Lambda^{n-1}+\cdots + g_{n-1}\Lambda. \label{pl}
\end{eqnarray}
Since $E_{ij}y^j=0$, see \cite[(6.11)]{Shen01}, it follows $\mathcal{E}^i_jy^j=0$ and hence $\det\left(\mathcal{E}^i_j\right)=0$. Therefore, $P(\Lambda)$ has no free term, $g_n=0$, the leading terms is $g_0=1$ and the other coefficients $g_a$, $a\in \{1,2...,n-1\}$, are $0$-homogeneous functions. 

We formulate now the main theorem of our work.
\begin{thm} \label{mainthm}
For a Finsler metric of vanishing $\chi$-curvature, the following $n-1$, $0$-homogeneous functions: 
\begin{eqnarray}
f_{a}=\operatorname{Tr}\left(\mathcal{E}^a\right), \quad a\in \{1,2,...,n-1\}, \label{fa}
\end{eqnarray}
and the coefficients functions $g_a$ of the characteristic polynomial \eqref{pl} are first integrals for the geodesic spray $G$, which means that $G(f_a)=0$ and $G(g_a)=0$.
\end{thm}
In formula \eqref{fa}, $\operatorname{Tr}$ is the semi-basic trace, \cite{Youssef94}, for the $a$-th power of the tensor \eqref{eij}. 
The two sets of first integrals $\{f_1,...,f_{n-1}\}$ and $\{g_1,...,g_{n-1}\}$ are related by Newton's identities \eqref{newton}.

An important class of Finsler metrics having the first integrals \eqref{fa} is given by Finsler metrics of constant flag curvature, see Corollary \ref{const}.
 
Using the expression \eqref{eCL} of the mean Berwald curvature in terms of the mean Cartan torsion and the mean Landsberg curvature, \cite[(6.38)]{Shen01}, we can reformulate the expression of the tensor \eqref{eij} and hence of the first integrals \eqref{fa} in terms of the mean Cartan torsion and the mean Landsberg curvature, see \eqref{eCL} and \eqref{f1}. In the $2$-dimensional case, the function $f_1$ becomes a function of the Cartan scalar and the Landsberg scalar, it was known from Berwald, \cite[(8.7)]{Berwald41}, and it is the first integral obtained by Foulon and Ruggiero in \cite[Theorem B]{FR16}.

We provide two proofs for the Theorem \ref{mainthm}.
For the first proof, where we show that the functions $f_a$ are first integrals, the key aspect that we will use relies on the fact that for a Finsler metric of vanishing $\chi$-curvature, the mean Berwald curvature is covariant constant along the geodesic flow (its dynamical covariant derivative vanishes). Therefore, the tensor \eqref{eij} and all its powers are covariant constant and consequently the functions \eqref{fa} are first integrals.

We will prove a more general result in Lemma \ref{lemcfi}, by providing necessary and sufficient conditions for $\nabla E_{ij}=0$. One of these conditions can be expressed in terms of the $\chi$-curvature and was obtained by Li and Shen in \cite[Theorem 1.2]{LS15}.

For the second proof of the Theorem \ref{mainthm}, where we show that the functions $g_a$ are first integrals, we will use a $1$-form $\alpha$, \eqref{alpha}, which is geodesically invariant (it is a dynamical symmetry) if and only if the Finsler metric has vanishing $\chi$-curvature, see Lemma \ref{lemchi}. The $1$-form $\alpha$ depends on the mean Cartan torsion and the mean Landsberg curvature and can be used to provide rigidity results in Finsler geometry. For a $k$-basic Finsler surfaces, this $1$-form has been introduced by Foulon and Ruggiero and it is invariant by the geodesic flow, \cite[Lemma 2.2]{FR16}. Using similar techniques developed in \cite{Bucataru22}, we will use the $0$-homogeneous $1$-form $\alpha$ to construct a set of $(2n-1)$-forms $\Omega_a$, $a\in \{1,...,n\}$, on the projective sphere bundle $SM$. These $(2n-1)$-forms are geodesically invariant and hence their proportionality factors $g_a$, which are also the coefficients of the characteristic polynomial \eqref{pl}, are first integrals if the the $\chi$-curvature vanishes. 

The second proof of Theorem \ref{mainthm} has been inspired from the techniques that were used in \cite{Bucataru22} to obtain a set of first integrals for geodesically equivalent Finsler metrics. For projectively equivalent Riemannian or Finslerian  metrics there are various methods that can be used to obtain first integrals, see \cite{MT98, Sarlet07, T99}.

\section{Finsler metrics and induced geometric structures.}

For a Finsler manifold $(M,F)$, the metric tensor \eqref{gij} and its regularity condition can be expressed in terms of the angular metric $h_{ij}$:
\begin{eqnarray*}
g_{ij}=h_{ij}+\frac{\partial F}{\partial y^i} \frac{\partial F}{\partial
  y^j}=h_{ij}+F_{y^i}F_{y^j}, \quad h_{ij}=F\frac{\partial^2
  F}{\partial y^i\partial y^j}=FF_{y^iy^j}.
\end{eqnarray*} 
The metric tensor $g_{ij}$ is non-degenerate if and only if the angular metric $h_{ij}$ has rank $n-1$.

For a Finsler manifold $(M, F)$, its geometric setting can be derived from its geodesic spray $G\in  \mathfrak{X}(T_0M)$. The spray $G$ is uniquely determined by the Euler-Lagrange equation:
\begin{eqnarray}
\mathcal{L}_Gd_JF^2-dF^2=0. \label{ELF2}
\end{eqnarray}
Here $J=\frac{\partial}{\partial y^i} \otimes dx^i$ is the tangent structure (vertical endomorphism), $d_J$ is the induced derivation of degree $1$ and $\mathcal{L}_G$ is the Lie derivative with respect to the geodesic spray $G$. Throughout this work we will use the Fr\"olicher-Nijenhuis theory of derivations to describe the geometric setting on a Finsler manifold \cite{BD09, Grifone72, GM00, SLK14, Youssef94}.

The geodesic spray $G$ is positively $2$-homogeneous, which means that it satisfies $[\mathcal{C}, G]=S$, for $\mathcal{C}=y^i\frac{\partial}{\partial y^i}$ the Liouvile (dilation) vector field. Locally, the spray $G$ can be expressed as:
\begin{eqnarray*}
G=y^i\frac{\partial}{\partial x^i} - 2G^i(x,y)\frac{\partial}{\partial y^i}, 
\end{eqnarray*} 
where the $2$-homogeneous functions $G^i$ can be deduced from the Euler-Lagrange equations \eqref{ELF2}, and they are given by:
\begin{eqnarray}
G^i(x,y)=\frac{1}{4}g^{ij}(x,y)\left(\frac{\partial^2 F^2}{\partial y^j\partial x^k}(x,y) y^k - \frac{\partial F^2}{\partial x^j}(x,y)\right). \label{gi}
\end{eqnarray} 
The spray $G$ induces a nonlinear connection (horizontal distribution) on $T_0M$, with the horizontal and vertical projectors given by \cite{Grifone72}:
\begin{eqnarray*}
h=\frac{1}{2}\left(\operatorname{Id}-[G,J]\right), \quad v=\frac{1}{2}\left(\operatorname{Id}+[G,J]\right).
\end{eqnarray*}
Locally, the two projectors can be expressed as follows:
\begin{eqnarray*}
h=\frac{\delta}{\delta x^i}\otimes dx^i, \  v=\frac{\partial}{\partial y^i}\otimes \delta y^i, \textrm{ \ where \ } \frac{\delta}{\delta x^i}=\frac{\partial}{\partial x^i} - N^j_i\frac{\partial}{\partial y^j}, \ \delta y^i = dy^i + N^i_j dx^j, \ N^j_i = \frac{\partial G^j}{\partial y^i}.
\end{eqnarray*}
The local frame $\{\frac{\delta}{\delta x^i}, \frac{\partial}{\partial y^i}\}$ is adapted to the horizontal and vertical distributions. The geodesic spray $G$ of a Finsler metric $F$ is a horizontal vector field, which means that $hG=G=y^i\frac{\delta}{\delta x^i}$.

From the Euler-Lagrange equations \eqref{ELF2} we obtain that the Finsler metric $F$ is geodesically invariant, $G(F)=0$, and $d_hF=0$, which means: 
\begin{eqnarray*}
dF=d_vF=\frac{\partial F}{\partial y^i}\delta y^i.
\end{eqnarray*} 

For a Finsler manifold $(M, F)$, we consider the projective sphere bundle $SM=T_0M/\mathbb{R}_+$ that can be identified with the indicatrix bundle $IM=F^{-1}\{1\}$. The Hilbert $1$ and $2$-forms: 
\begin{eqnarray*}
d_JF=\frac{\partial F}{\partial y^i}dx^i, \quad
  dd_JF= \frac{1}{F}h_{ij} \delta y^i\wedge dx^j 
\end{eqnarray*}
can be restricted to the sphere bundle $SM$, $d_JF\in \Lambda^1(SM)$, $dd_JF\in \Lambda^2(SM)$, since they are $0$-homogeneous, $i_{\mathcal{C}}d_JF=0$ and $i_{ \mathcal{C}}dd_JF=0$. Moreover, 
\begin{eqnarray*}
\operatorname{rank}(dd_JF)=2 \operatorname{rank}(h_{ij})=2(n-1).
\end{eqnarray*}

The regularity condition of a Finsler metric $F$ ensures that its Hilbert $1$-form is a contact structure on $SM$, which means that $d_JF\wedge \left(dd_JF\right)^{(n-1)}\neq 0$. Therefore, a Finsler metric induces a canonical volume form on $SM$:
\begin{eqnarray}
\Omega_F=\frac{(-1)^{(n-1)(n-2)/2}}{(n-1)!}d_JF\wedge
  \left(dd_J F\right)^{(n-1)} \in \Lambda^{2n-1}(SM). \label{on} 
\end{eqnarray}

The regularity condition of a Finsler metric $F$ ensures also that the following $2$-form is a symplectic form on $T_0M$:
\begin{eqnarray}
\frac{1}{2}dd_JF^2=g_{ij}\delta y^i\wedge dx^j = Fdd_JF+dF\wedge d_JF = Fdd_JF+d_vF\wedge d_JF. \label{ddjf2}
\end{eqnarray}
The symplectic form \eqref{ddjf2} induces a canonical volume form on $T_0M$:
\begin{eqnarray}
\Omega'_F=\frac{(-1)^{n(n+1)/2}}{n!}\left(\frac{1}{2}dd_JF^2\right)^{(n)}=\det g \ dx\wedge dy  \in \Lambda^{2n}(T_0M). \label{o'n}
\end{eqnarray}
Using formulae \eqref{on}, \eqref{ddjf2} and \eqref{o'n}, we have that  the  two volume forms $\Omega_F \in \Lambda^{2n-1}(SM)$ and $\Omega'_F \in \Lambda^{2n}(T_0M)$ are related by:
\begin{eqnarray*}
\Omega'_F=(-1)^{n-1} F^{n-1} dF\wedge \Omega_F = (-1)^{n-1} F^{n-1} d_vF\wedge \Omega_F, \quad \Omega_F = \frac{(-1)^{n-1}}{F^{n}}i_{\mathcal{C}}\Omega'_F. 
\end{eqnarray*}

Among the many curvature tensors that exist in Finsler geometry, we recall first the Jacobi endomorphism and the curvature of the nonlinear connection.

The Jacobi endomorphism is the $(1,1)$-type tensor field:
\begin{eqnarray*}
\Phi=v\circ [G, h]= R^i_j\frac{\partial}{\partial y^i}\otimes dx^j, \quad R^i_j=2\frac{\partial G^i}{\partial x^j} - G\left(N^i_j\right) - N^i_kN^k_j.
\end{eqnarray*} 

The curvature of the nonlinear connection can be defined as the Nijenhuis tensor of the horizontal projector $h$, which gives the obstruction to the integrability of the horizontal distribution: 
\begin{eqnarray*}
R=\frac{1}{2}[h,h]=\frac{1}{2}R^i_{jk}\frac{\partial}{\partial y^i}\otimes dx^k\wedge dx^j, \quad R^i_{jk}= \frac{\delta N^i_j}{\delta x^k} - \frac{\delta N^i_k}{\delta x^j}.
\end{eqnarray*}

A Finsler metric has scalar flag curvature if there exists a function $\kappa\in C^{\infty}(T_0M)$ such that its Jacobi endomorphism has the form:
\begin{eqnarray*}
\Phi &  = & \kappa\left( F^2 J - Fd_JF\otimes \mathcal{C} \right) \Longleftrightarrow R^i_j=\kappa \left(\delta^i_j g_{sl} - \delta^i_s g_{jl}\right) y^sy^l.
\end{eqnarray*}
When the $0$-homogeneous function $\kappa$ is constant, we say that the Finsler metric has constant flag curvature. A Schur's Lemma for Finsler metrics is true: in dimension $n>2$, if the scalar flag curvature $\kappa$ does not depend on the fibre coordinates $y$ then $\kappa$ is constant and the Finsler metric has constant flag curvature, \cite[Theorem 3.1]{BC20}, \cite[Theorem 9.4.1]{SLK14}.   

While the Jacobi endomorphism and the curvature of the nonlinear connection have Riemannian correspondents (depend quadratic and respectively linear on the Riemannian curvature tensor), the Berwald curvature $B^i_{jkl}$ and the mean Berwald curvature $E_{ij}$ are purely non-Riemannian tensors, and they are given by:
\begin{eqnarray}
B^i_{jkl}=\frac{\partial^3 G^i}{\partial y^j \partial y^k \partial y^l}, \quad E_{ij}=\frac{1}{2}B^k_{ijk}=\frac{1}{2} \frac{\partial^3 G^k}{\partial y^i \partial y^j \partial y^k}. \label{be}
\end{eqnarray}
The mean Berwald curvature can be expressed in terms of the $S$-function, introduced by Shen, \cite[\S 5.2]{Shen01}. For a fixed vertically invariant volume form $\sigma(x)dx\wedge dy$ on TM, \cite[p.490]{SLK14}, we consider the $S$-\emph{function} and the \emph{distortion} $\tau$: 
\begin{eqnarray}
S=G(\tau), \quad \tau=\frac{1}{2}\ln \frac{\det g}{\sigma}.  \label{sgt} 
\end{eqnarray} 
The mean Berwald curvature can be expressed as follows, \cite[(6.13)]{Shen01}:
\begin{eqnarray}
E_{ij}=\frac{1}{2}\frac{\partial^2 S}{\partial y^i \partial y^j}, \ \textrm{hence \ } 2E_{ij}\delta y^i \wedge dx^j = d_vd_JS. \label{es}
\end{eqnarray}

Another important non-Riemannian quantity, the $\chi$-curvature, was introduced by Shen, \cite[(1.10)]{Shen13}. This can be defined as:
\begin{eqnarray}
\chi=\frac{1}{2}\delta_GS, \quad \textrm{where \ } \delta_GS=\mathcal{L}_Gd_JS - dS = \left\{G\left(\frac{\partial S}{\partial y^i}\right) - \frac{\partial S}{\partial x^i}\right\} dx^i, \label{chi}
\end{eqnarray}  
is the Euler-Lagrange semi-basic $1$-form associated to the spray $G$ and the $S$-function. 

The dynamical covariant derivative $\nabla$ is a derivation of degree zero along the geodesic spray. It can be defined through its action on functions and vector fields and then extended to arbitrary tensor fields and forms, \cite{BD09}:
\begin{eqnarray*}
\nabla f=G(f), \forall f \in  C^{\infty}(T_0M), \quad \nabla X=h[G, hX] + v[G, vX], \forall X \in \mathfrak{X}(T_0M).
\end{eqnarray*} 
For a Finsler structure $F$, the dynamical covariant derivative of its metric tensor vanishes, $\nabla g_{ij}=0$, \cite[(44)]{BD09} and the same happens to the Finsler function $\nabla F = G(F)=0$. Therefore, the tensor \eqref{eij} and all its geometric invariants are covariant constant along the geodesic flow if $\nabla E_{ij}=0$.

\section{Proofs of Theorem \ref{mainthm}.}

The most important aspects that we will use in the proof of Theorem \ref{mainthm} are based on the geodesic invariance of some geometric structures on Finsler manifolds with vanishing $\chi$-curvature. For the first proof we will use the fact that the mean Berwald curvature is covariant constant, $\nabla E_{ij}=0$, and therefore the tensor \eqref{eij} is covariant constant along the geodesic flow if the $\chi$-curvature vanishes. For the second proof we show in Lemma \ref{lemchi} that a $1$-form \eqref{alpha} is geodesically invariant if and only if the $\chi$-curvature vanishes. We use the form $\alpha$ to construct a set of $(2n-1)$-forms on the projective sphere bundle $SM$ that are geodesically invariant, and whose proportionality factors are first integrals.

\subsection{First proof of Theorem \ref{mainthm}} 

For the first proof, we provide in Lemma \ref{lemcfi} necessary and sufficient conditions for the mean Berwald curvature, and hence for the tensor \eqref{eij}, to be covariant constant along the geodesic flow. These conditions will allow to obtain various characterisations for a class of Finsler metrics (that includes Finsler metrics of constant flag curvature) which have the functions \eqref{fa} as a set of first integrals. 
 
\begin{lem} \label{lemcfi}
On a Finsler manifold $(M, F)$, the following conditions are equivalent:
\begin{itemize}
\item[i)] $\nabla E_{ij}=0$;
\item[ii)] $d_hd_JS$ is a basic $2$-form;
\item[iii)] there exists a basic $2$-form $\omega \in \Lambda^2(M)$ such that $2\chi = i_G\omega$.
\end{itemize}
\end{lem}

\begin{proof}
We prove i) $\Longrightarrow$ ii). Using the expression \eqref{es} of the mean Berwald curvature in terms of the $S$-function, the assumption  $\nabla E_{ij}=0$ reads  $\nabla d_vd_JS=0$. To this equation we apply derivation $d_J$ and use the commutation rule \cite[(2.11)]{BC15}, which gives:
\begin{eqnarray*}
0=d_J\nabla d_vd_JS=\nabla d_Jd_vd_JS + d_hd_vd_JS + 2i_Rd_vd_JS = -d_vd_hd_JS.
\end{eqnarray*}
Here we made use of the following $d_Jd_v=-d_vd_J$ and $d^2_J=0$ which gives $d_Jd_vd_JS=0$. We also used that $2R=[h,h]= [h, \operatorname{Id}-v]=-[h,v]$ and hence $d_hd_vd_JS+d_vd_hd_JS=-2d_Rd_JS=-2i_Rd_vd_JS$. The fact that $d_vd_hd_JS=0$ implies that $d_hd_JS$ is a basic $2$-form ($d_hd_JS$ is already a semi-basic $2$-form).

We prove ii)) $\Longrightarrow$ iii). We assume that there is a basic $2$-form $\omega \in \Lambda^2(M)$ such that $d_hd_JS=\omega$. We apply the inner product $i_G$ to both sides of this equality and use the commutation rule for the two derivations $i_G$ and $d_h$, \cite[p.205]{GM00}, 
\begin{eqnarray*}
i_G\omega & = & i_Gd_hd_JS=-d_hi_Gd_JS+ \mathcal{L}_{hG}d_JS+d_{[h,G]}d_JS \\ & = & -d_hS +\mathcal{L}_Gd_JS-d_vS=\mathcal{L}_Gd_JS - dS \stackrel{\eqref{chi}}{=} 2\chi.
\end{eqnarray*}
The equality of the first and the last terms is the condition iii).
 
We prove the last implication iii) $\Longrightarrow$ i). Formula \eqref{chi} allows to write our assumption as $\delta_GS=i_G\omega$. Using formulae \cite[(2.13)]{BC15}, which provide alternative expressions for the Euler-Lagrange $1$-form, we can rewrite this assumption as $\nabla d_JS - d_hS=i_G\omega$. If we apply the derivation $d_J$ to both sides of the last formula and use again the commutation rule \cite[(2.11)]{BC15}, we obtain:
\begin{eqnarray*}
\nabla d_J^2S + d_hd_JS + 2i_Rd_JS - d_Jd_hS=d_Ji_G\omega \Longleftrightarrow d_hd_JS=\omega.
\end{eqnarray*}  
We apply the differential operator $d$ to both sides of our assumption and then:
\begin{eqnarray*}
\mathcal{L}_Gdd_JS=di_G\omega \Longleftrightarrow \mathcal{L}_Gd_hd_JS + \mathcal{L}_Gd_vd_JS =\mathcal{L}_G\omega - i_Gd\omega \Longleftrightarrow \mathcal{L}_Gd_vd_JS =- i_Gd\omega.
\end{eqnarray*}
We replace the Lie derivative $\mathcal{L}_G$ in terms of the covariant derivative $\nabla$, \cite[(27)]{BD09} and get:
\begin{eqnarray*}
\nabla d_vd_JS - d_{\Phi}d_JS = -i_Gd\omega.
\end{eqnarray*}  
In the above equation, we separate the semi-basic part from the non-semi-basic part. This implies $\nabla d_vd_JS=0$ (the non-semi-basic part), which in view of formula \eqref{es} gives $\nabla E_{ij}=0$ and hence the condition i) is true. We also obtain the equality of the semi-basic parts $d_{\Phi}d_JS= i_Gd\omega$. \end{proof}

The equivalence of i) and iii) has been established by Li and Shen in \cite[Theorem 1.2]{LS15}.

The first proof of Theorem \ref{mainthm} follows immediately from Lemma \ref{lemcfi}. For a Finsler metric of vanishing $\chi$-curvature, formula \eqref{chi} and the condition iii) of Lemma \ref{lemcfi} show that the assumption $\chi=0$ implies the condition i) of Lemma \ref{lemcfi}, $\nabla E_{ij}=0$. Therefore, the tensor \eqref{eij} and all its powers are covariant constant, $\nabla \mathcal{E}^a=0$, $\forall a\in \mathbb{N}$. Since the dynamical covariant derivative commutes with the semi-basic trace, it follows that
\begin{eqnarray*}
G(f_a)=\nabla (f_a) = \nabla \left(\operatorname{Tr} \mathcal{E}^a\right)= \operatorname{Tr} \left(\nabla  \mathcal{E}^a\right)=0, \forall a\in \{1,...,n-1\},  
\end{eqnarray*} 
which means that all the $0$-homogeneous functions $f_a =\operatorname{Tr} \mathcal{E}^a$ are first integrals for the geodesic spray $G$.

\subsection{Second proof of Theorem \ref{mainthm}} 
For a Finsler metric $F$, its reducibility to a Riemannian metric is encoded into its mean Cartan torsion, \cite[Deicke's Theorem]{Deicke53}:
\begin{eqnarray*}
I_k=\frac{1}{2}g^{ij}\frac{\partial g_{ij}}{\partial y^k}=\frac{\partial}{\partial y^k}\left(\ln \sqrt{\det g}\right) = \frac{\partial \tau}{\partial y^k}, \quad I=I_kdx^k = d_J\left(\ln \sqrt{\det g}\right) = d_J\tau.
\end{eqnarray*}
The $0$-homogeneity of the distortion function $\tau$ implies $i_GI=y^iI_i=\frac{\partial \tau}{\partial y^i}y^i=0$.

An important ingredient we will use for the second proof of Theorem \ref{mainthm} is the following $1$-form:
\begin{eqnarray} \label{alpha}
\alpha & = & i_{[J,G]}\mathcal{L}_GI = \nabla I_k dx^k - I_k \delta y^k = \nabla d_J \tau  - d_v\tau \\
& = & d_J\nabla \tau - d_h\tau - d_v\tau = d_J\nabla \tau - d\tau = d_JS - d\tau, \nonumber
\end{eqnarray}
which in the $2$-dimensional case has been introduced in \cite[\S 2]{FR16}. We can easily see that $i_{\mathcal{C}}\alpha = y^iI_i=0$ and $i_G\alpha = y^i\nabla I_i=0$. 

\begin{lem} \label{lemchi}
On a Finsler manifold $(M,F)$, the $\chi$-curvature vanishes if and only if one of the following equivalent conditions is satisfied: 
\begin{itemize}
\item[i)] $\mathcal{L}_G\alpha=0$: the $1$-form $\alpha$ is geodesically invariant (it is a dual symmetry);
\item[ii)] $d_hd_JS=0$: the $S$-function is a Hamel function;
\item[iii)] $d\alpha=2E_{ij}\delta y^i \wedge dx^j$: the horizontal and vertical distributions are Lagrangian distributions for $d\alpha$. 
\end{itemize}
\end{lem}
\begin{proof}
Formulae \eqref{alpha} allow to express the geodesic variation of the $1$-form $\alpha$ in terms of the $\chi$-curvature:
\begin{eqnarray}
\mathcal{L}_G\alpha=\mathcal{L}_Gd_JS-\mathcal{L}_Gd\tau = \mathcal{L}_Gd_JS-d\mathcal{L}_G\tau \stackrel{\eqref{sgt}}{=}\mathcal{L}_Gd_JS-dS \stackrel{\eqref{chi}}{=}2\chi,  \label{lgalpha} 
\end{eqnarray}
and hence $\chi=0$ if and only if $\mathcal{L}_G\alpha=0$.

The expression of the $\chi$-curvature, given by formula \eqref{chi}, can be reformulated using the dynamical covariant derivative, \cite[(2.13)]{BC15}, as follows:
\begin{eqnarray*}
2\chi=\nabla d_JS - d_hS.
\end{eqnarray*}
If we apply to both sides of this formula the derivation $d_J$ and use the commutation rule  \cite[(2.11)]{BC15}, we obtain:
\begin{eqnarray*}
2d_J\chi=d_J\nabla d_JS - d_Jd_hS = \nabla d^2_JS + d_hd_JS + 2i_Rd_JS - d_Jd_hS = 2d_hd_JS.
\end{eqnarray*}
Above formulae give $d_J\chi=d_hd_JS$. We apply to these semi-basic $2$-forms the inner product $i_G$. Therefore, if we use the commutation rule for the derivations $i_G$ and $d_h$, \cite[p.205]{GM00}, we have:
\begin{eqnarray*}
i_Gd_J\chi=i_Gd_hd_JS=-d_hi_GS+\mathcal{L}_{hG}d_JS + d_{[h,G]}S = -d_hS+\mathcal{L}_{G}d_JS - d_vS = \mathcal{L}_{G}d_JS - dS =2\chi.  
\end{eqnarray*}  
The last formulae allow to conclude that $\chi=0$ if and only if $d_J\chi=0$ if and only if $d_hd_JS=0$.

From formulae \eqref{alpha} and \eqref{es} we obtain: 
\begin{eqnarray*}
d\alpha=dd_JS=d_hd_JS+d_vd_JS=d_hd_JS+2E_{ij}\delta y^i\wedge dx^j.
\end{eqnarray*}
Therefore, $d\alpha=2E_{ij}\delta y^i\wedge dx^j$ if and only if $d_hd_JS=0$ if and only if $\chi=0$.
\end{proof}
For a Finsler surface, the $\chi$-curvature vanishes if and only if we have a $k$-basic Finsler structure, \cite[Theorem 1.1]{LS18}. Therefore, Lemma \ref{lemchi} allows to recover the geodesic invariance of the $1$-form $\alpha$ for $k$-basic Finsler surfaces obtained in \cite[Lemma 2.2]{FR16}. 

For the second proof of Theorem \ref{mainthm} we use the characterizations established in Lemma \ref{lemchi} for Finsler metrics of vanishing $\chi$-curvature. 

The $1$-form $\alpha$ is $0$-homogeneous and satisfies $i_{\mathcal{C}}\alpha=0$, which means that it can be viewed as a form on $SM$, $\alpha\in \Lambda^1(SM)$. Using similar ideas from \cite[\S 3]{Bucataru22}, we will use this form and the Hilbert $1$ and $2$-forms $d_JF\in \Lambda^1(SM)$ and $dd_JF\in \Lambda^2(SM)$ to construct, for each $a\in \{1,...,n\}$, an invariant $(2n-1)$-form on $SM$:
\begin{eqnarray}
\Omega_a=\frac{(-1)^{(n-1)(n-2)/2}}{(a-1)! (n-a)!}d_JF\wedge
  \left(dd_J F\right)^{(a-1)} \wedge
  \left(d\alpha\right)^{(n-a)} \in \Lambda^{2n-1}(SM). \label{oa} 
\end{eqnarray}
We can seen that $\Omega_n=\Omega_F$ is the canonical volume \eqref{on} of the contact manifold $(SM, d_JF)$. 

The Reeb vector field of the contact manifold $(SM, d_JF)$ is the normalized geodesic field $G/F$. Using \cite[Lemma 4.1]{Bucataru22} it follows that the Hilbert $1$ and $2$-forms $d_JF$ and $dd_JF$ are preserved by the action of the Reeb vector field $G/F$:
\begin{eqnarray*}
\mathcal{L}_{G/F}d_JF =\frac{1}{F}\left(\mathcal{L}_{G}d_JF - dF\right)  = 0 \ \textrm{and} \ \mathcal{L}_{G/F}dd_JF=0.
\end{eqnarray*}
Since $i_G\alpha=y^i\nabla I_i = 0$, we also have:
\begin{eqnarray*}
\mathcal{L}_{G/F}\alpha =\frac{1}{F}\mathcal{L}_{G}\alpha + i_G\alpha \ d\left(\frac{1}{F}\right) = \frac{1}{F}\mathcal{L}_{G}\alpha \stackrel{\eqref{lgalpha}}{=} 0,
\end{eqnarray*}
which means that the form $\alpha$ is also invariant by the action of the Reeb vector field. Consequently, all the $(2n-1)$-forms \eqref{oa} are invariant as well: 
\begin{eqnarray*}
\mathcal{L}_{G/F}\Omega_a=0, \quad \forall a\in \{1,...,n\}.
\end{eqnarray*}
It follows that the proportionality factors of these $(2n-1)$-forms:
\begin{eqnarray}
\Omega_a=g_{n-a}\Omega_n, \quad \forall a\in \{1,...,n-1\}, \label{ga}
\end{eqnarray}
are also preserved by the action of the Reeb vector field. Therefore, the functions $g_a\in C^{\infty}(SM)$, $a\in \{1,...,n-1\}$, are first integrals for the Reeb vector field and hence for the geodesic spray, $G(g_a)=0$.

We prove now that the first integrals $g_a$, given by formula \eqref{ga}, are the coefficients of the characteristic polynomial \eqref{pl}.

Together with the $(2n-1)$ forms \eqref{oa} on $SM$, we consider the following $2n$-forms on $T_0M$:
\begin{eqnarray}
\Omega'_a=(-1)^{n-1} F^{n-1} dF\wedge \Omega_a = (-1)^{n-1} F^{n-1} d_vF\wedge \Omega_a , \quad a\in \{1,...,n\}. \label{o'a}
\end{eqnarray} 
It is easy to see that $\Omega'_n=\Omega'_F$ is the volume form \eqref{o'n} and the functions $g_a$ are uniquelly determined also by $\Omega'_a=g_a\Omega'_n$, $\forall a\in \{1,...,n-1\}$. First, we can write the characteristic polynomial \eqref{pl} as follows:
\begin{eqnarray*}
P(\Lambda)=\det\left(g^{ik}2FE_{kj}+ \Lambda \delta^i_j\right) = \frac{1}{{\det g}}\det\left(2FE_{kj}+ \Lambda g_{kj}\right). 
\end{eqnarray*}
Therefore, we have:
\begin{eqnarray*}
P(\Lambda)\Omega'_n & \stackrel{\eqref{o'n}}{=} & \frac{1}{{\det g}}\det\left(2FE_{kj}+ \Lambda g_{kj}\right) \det g \ dx\wedge dy \\
& = & \frac{(-1)^{n(n+1)/2}}{n!}\left(\left(2FE_{kj}+\Lambda g_{kj}\right)\delta y^k \wedge dx^j \right)^{(n)} \\
&\stackrel{\eqref{ddjf2}}{ =} & \frac{(-1)^{n(n+1)/2}}{n!}\left(2FE_{kj} \delta y^k \wedge dx^j +\Lambda \frac{1}{2}dd_JF^2\right)^{(n)} \\
&\stackrel{(a)}{ = } & \frac{(-1)^{n(n+1)/2}}{n!}\left(Fd\alpha +\Lambda \left(Fdd_JF + dF\wedge d_JF\right)\right)^{(n)} \\
&\stackrel{(b)}{ = } & \frac{(-1)^{n(n+1)/2}}{n!}\left( \sum_{a=0}^n \frac{n!}{a!(n-a)!} \Lambda^a \left(Fdd_JF + dF\wedge d_JF\right)^{(a)} \wedge (Fd\alpha)^{(n-a)}\right) \\ 
&\stackrel{(c)}{ = } & (-1)^{n(n+1)/2} \left( \sum_{a=1}^n \frac{\Lambda^a}{(a-1)!(n-a)!} dF\wedge d_JF \wedge \left(Fdd_JF\right)^{(a-1)} \wedge \left(Fd\alpha\right)^{(n-a)}\right) \\
& = & (-1)^{n-1} F^{n-1} dF \wedge \left( \sum_{a=1}^n \Lambda^a\frac{(-1)^{(n-1)(n-2)/2}}{(a-1)!(n-a)!} d_JF \wedge \left(dd_JF\right)^{(a-1)} \wedge \left(d\alpha\right)^{(n-a)} \right) \\
& \stackrel{\eqref{oa}}= & (-1)^{n-1} F^{n-1} dF \wedge \left( \sum_{a=1}^n \Lambda^a \Omega_a \right)
 \stackrel{\eqref{ga}}{=} (-1)^{n-1} F^{n-1} dF \wedge \left(\Lambda^n+ \sum_{a=1}^{n-1} \Lambda^a g_{n-a}\right) \Omega_n  \\ & \stackrel{\eqref{o'a}}{=}&  \left(\Lambda^n+ \sum_{a=1}^{n-1} g_{n-a}\Lambda^a \right) \Omega'_n.
\end{eqnarray*}
It follows that the proportionality functions \eqref{ga} are the coefficients of the characteristic polynomial \eqref{pl}. 

In the above calculations, we used the following arguments:

(a) For vanishing $\chi$-curvature, the expression $d\alpha = 2E_{kj} \delta y^k \wedge dx^j$ is given by iii) of Lemma \ref{lemchi}. We also note here that $\operatorname{rank}(d\alpha)\leq 2\operatorname{rank}(E_{kj}) \leq 2(n-1)$. 

(b) Since $2$-forms commute under the wedge product, we can use a binomial Newton-type formula for the wedge power.

(c) We use again a binomial Newton-type formula for the wedge power to expand 
\begin{eqnarray*}
\left(Fdd_JF + dF\wedge d_JF\right)^{(a)}=\left(Fdd_JF\right)^{(a)} + a \  dF\wedge d_JF \wedge \left(Fdd_JF\right)^{(a-1)}. 
\end{eqnarray*}
However, the following $2n$-form vanishes, $\left(Fdd_JF\right)^{(a)}\wedge \left(Fd\alpha\right)^{(n-a)} =0$, since $\operatorname{rank}(dd_JF)=2(n-1)$ and  $\operatorname{rank}(d\alpha)\leq 2(n-1)$. 

The two sets of first integrals $\{f_1,...,f_{n-1}\}$ and $\{g_1,...,g_{n-1}\}$ from Theorem \ref{mainthm} can be related using Newton's identities, \cite{K95}: 
\begin{eqnarray}
g_{a}=\frac{1}{a!}
\begin{vmatrix}
f_{1} & 1 & 0 & \cdots & \cdots & \cdots & 0\\
f_{2} & f_{1} & 2 & 0 & \cdots & \cdots & 0\\
f_{3} & f_{2} & f_{1} & 3 & 0 & \cdots & 0\\
\cdots & \cdots & \cdots & \cdots & \cdots & \cdots & \cdots\\
f_{a-1} & f_{a-2} & f_{a-3} & \cdots & \cdots & f_{1} & a-1\\
f_{a} & f_{a-1} & f_{a-2} & \cdots & \cdots & f_{2} & f_{1}
\end{vmatrix}, \quad \forall a\in \{1,...,n-1\}.  \label{newton}
\end{eqnarray}
It is easy to see that $f_1=g_1=\operatorname{Tr}\left(\mathcal{E}\right)$, and general linear recurrence formulae between the sets of first integrals $\{f_1,...,f_{a}\}$ and $\{g_1,...,g_{a}\}$ can be established, for any $a \in \{1,..,n-1\}$, \cite{K95}.

\section{Particular cases, alternative expressions for the first integrals and examples}

\subsection{Particular cases.} Since for Finsler metrics of constant curvature, the  $\chi$-curvature vanishes, \cite[Theorem 1.1]{LS18}, we obtain the following corollary.  

\begin{cor} \label{const}
Consider $F$ a Finsler metric of constant curvature. Then the functions \eqref{fa} and the coefficients of the characteristic polynomial \eqref{pl} are constant along the geodesics of the Finsler metric $F$. 
\end{cor}

We will show now that the first integral $g_{n-1}$ coincides with the first integral \cite[(1.1)]{BCC21}. Indeed, using formulae \eqref{oa} and \eqref{o'a}, for $a=1$, we have:
\begin{eqnarray*}
\Omega'_1 & = & \frac{(-1)^{n(n+1)/2}}{(n-1)!} dF\wedge d_JF \wedge \left(Fd\alpha\right)^{(n-1)} \\
& = &  \frac{(-1)^{n(n+1)/2}}{n!} \left( d_vF\wedge d_JF + Fd\alpha\right)^{(n)} \\
& = & \frac{(-1)^{n(n+1)/2}}{n!} \left(  \left(2FE_{ij}+\frac{\partial F}{\partial y^i} \frac{\partial F}{\partial y^j}\right) \delta y^i \wedge dx^j \right)^{(n)} \\
& = & \det\left(2FE_{ij}+\frac{\partial F}{\partial y^i} \frac{\partial F}{\partial y^j}\right) dx\wedge dy =\frac{1}{\det g} \det\left(2FE_{ij}+\frac{\partial F}{\partial y^i} \frac{\partial F}{\partial y^j}\right) \Omega'_n.   
\end{eqnarray*} 
These arguments and formula \eqref{oa} for $a=1$ give the first integral:
\begin{eqnarray*}
g_{n-1}=\frac{1}{\det g} \det\left(2FE_{ij}+\frac{\partial F}{\partial y^i} \frac{\partial F}{\partial y^j}\right) = \frac{-1}{\det g}  \begin{vmatrix}
2FE_{ij} & \displaystyle\frac{\partial
  F}{\partial y^i} \vspace{2mm} \\
\displaystyle\frac{\partial F}{\partial y^j} & 0 
\end{vmatrix},
\end{eqnarray*}
which is the first integral \cite[(1.1)]{BCC21}.

\subsection{Alternative expressions for the first integrals.}
The mean Berwald curvature can be expressed in terms of the mean Cartan torsion and the mean Landsberg curvature, \cite[Lemma 6.2.4]{Shen01}:
\begin{eqnarray}
E_{ij}=\frac{1}{2}\left\{I_{j;i}+J_{i\cdot j}\right\}, \label{eCL}
\end{eqnarray}
where: $J_i=\nabla I_i$ is the mean Landsberg curvature, $J_{i\cdot j}=\frac{\partial J_i}{\partial y^j}$ is the vertical covariant derivative of the mean Landsberg curvature and 
\begin{eqnarray*}
I_{j;i}=\frac{\delta I_j}{\delta x^i} - I_l\frac{\partial^2 G^l}{\partial y^i \partial y^j}
\end{eqnarray*} 
is the horizontal covariant derivative of the mean Cartan torsion with respect to the Berwald connection.

It follows that the function 
\begin{eqnarray}
g_1=f_1=\frac{1}{2}g^{ij}\left\{I_{j;i}+J_{i\cdot j}\right\} \label{f1}
\end{eqnarray}
is a first integral for a Finsler metric of vanishing $\chi$-curvature.
For a $k$-basic Finsler surface, the first integral \eqref{f1} reduces to the first integral obtained in \cite[Theorem B]{FR16} as a function of the Cartan scalar $I$ and the Landsberg scalar $J$.

\subsection{Example.}

In this section we present an example of a Finsler metrics of constant curvature and the corresponding first integrals obtained using Theorem \ref{mainthm}. Due to the complexity of the calculations, these were obtained in dimension $3$ using a software package.

Consider $B^3$ the open unit ball in $\mathbb{R}^3$ and $F: B^3\times \mathbb{R^*}^3 \to \mathbb{R}$ the Berwald metric:
\begin{eqnarray*}
F(x,y)=\frac{\left(\sqrt{|y|^2-(|x|^2|y|^2-\langle x, y\rangle^2)}+\langle x, y \rangle\right)^2}{\left(1-|x|^2\right)^2\sqrt{|y|^2-(|x|^2|y|^2-\langle x, y\rangle^2)}}.
\end{eqnarray*}
The Finsler metric $F$ is $R$-flat and projectively flat, with the projective factor:
\begin{eqnarray*}
P(x,y)=\frac{\sqrt{|y|^2-(|x|^2|y|^2-\langle x, y\rangle^2)}}{1-|x|^2}+\frac{\langle x, y \rangle}{1-|x|^2},\ y\in T_xB^3,
\end{eqnarray*}
being the Funk metric, \cite[(2.48)]{Shen01}. The spray coefficients \eqref{gi} are given by $G^i(x,y)=P(x,y)y^i$ and hence the mean Berwald curvature \eqref{be} is give by:
\begin{eqnarray*}
E_{ij} = \frac{1}{2} \frac{\partial^3 G^l}{\partial y^i \partial y^j \partial y^l} = \frac{n+1}{2} \frac{\partial^2 P}{\partial y^i \partial y^j}.
\end{eqnarray*}
Therefore, the tensor \eqref{eij} is given by
\begin{eqnarray*}
\mathcal{E}^i_j=(n+1)Fg^{ik}\frac{\partial^2P}{\partial y^k \partial y^j},
\end{eqnarray*}
and its characteristic polynomial \eqref{pl} gives the following $2$ first integrals:
\begin{equation*}
\begin{aligned}
g_1&=\dfrac{\left(-2\langle x,y\rangle\sqrt{|y|^2-|x|^2|y|^2+\langle x,y\rangle^2}-2\langle x,y\rangle^2-|y|^2(1-|x|^2)\right)\cdot\left(|y|^2-|x|^2|y|^2+\langle x,y\rangle^2\right)^2)}{8\left(\frac{1}{2}|y|^2+|x|^2|y|^2-2\langle x,y\rangle^2\right)|y|^2\left(\langle x,y\rangle+\sqrt{|y|^2-|x|^2|y|^2+\langle x,y\rangle^2}\right)^2},
\end{aligned}
\end{equation*}
\begin{equation*}
\begin{aligned}
g_2&=1+\dfrac{\left(2|y|^2+|x|^2|y|^2-\langle x,y\rangle^2\right)\cdot\left(|x|^2|y|^2-\langle x,y\rangle^2\right)}{2|y|^2\left(-\frac{1}{2}|y|^2+|y|^2(1+|x|^2)-\langle x,y\rangle^2\right)}.
\end{aligned}
\end{equation*}
To compute the Poisson bracket of these two first integrals, we use the software package, to obtain:
\begin{eqnarray*}
\nonumber \left\{g_1, g_2\right\}  = \frac{1}{2} g^{ij} \dfrac{\partial g_1}{\partial y^j} \dfrac{\delta g_2}{\delta x^i}  - \frac{1}{2} g^{ij}\dfrac{\partial g_2}{\partial y^j} \dfrac{\delta g_1}{\delta x^i}=0. 
\end{eqnarray*} 
Therefore, the two first integrals $g_1$ and $g_2$ Poisson commute with respect to the flat metric and also with respect to the Berwald metric.

%\subsection*{Acknowledgements} 

\end{document}